\documentclass{article}

\usepackage{graphicx}
\usepackage{graphics}
\usepackage{epsfig}
\usepackage{psfrag}   
\usepackage{subfigure} 
\usepackage{amsmath}
\usepackage{amssymb}
\usepackage{amsthm}
\usepackage{deleq}   
\usepackage{type1cm} 

\RequirePackage[left=1in,right=1in,top=1in,bottom=1in]{geometry}

\newtheorem{definition}{Definition}
\newtheorem{remark}{Remark}
\newtheorem{theorem}{Theorem}
\newtheorem{lemma}{Lemma}

\newtheorem{proposition}{Proposition}
\newtheorem{counterex}{Counterexample}

\begin{document}

\title{Finite Alphabet Control of Logistic Networks with Discrete Uncertainty}
\author{Danielle~C.~Tarraf\thanks{D. C. Tarraf is with the Department of Electrical \& Computer Engineering at Johns 
Hopkins University, Baltimore, MD 21218 (dtarraf@jhu.edu).} 
\hspace*{5pt}and 
Dario~Bauso\thanks{D. Bauso is with the Dipartimento di Ingegneria Chimica, Gestionale, 
Informatica e Meccanica, Universit\'a di Palermo  (dario.bauso@unipa.it).}
}

\maketitle

\begin{abstract}
We consider logistic networks
in which the control and disturbance inputs take values in finite sets.
We derive a necessary and sufficient condition for the existence of robustly control invariant (hyperbox) sets.
We show that a stronger version of this condition is sufficient to guarantee robust global attractivity, 
and we construct a counterexample demonstrating that it is not necessary.
Being constructive, our proofs of sufficiency
allow us to extract the corresponding robust control laws
and to establish the invariance of certain sets.
Finally, we highlight parallels between our results and existing results in the literature,
and we conclude our study with two simple illustrative examples.
\end{abstract}

\section{Introduction}
\label{Sec:Introduction}
Production networks, distribution networks and transportation
networks are all instances of logistic networks,
essentially dynamic network flow problems in which the inventories,
controlled flows, supplies and demands evolve in time.
The dynamics of these systems are linear: 
The state of the system represents accumulated products in production stages, 
inventory stored in warehouses,
or commodities available at hubs. 
The control inputs represent the controlled flows between these production stages, 
warehouses or hubs, 
such as production processes or transportation routes.
The disturbance inputs represent fluctuating supplies of raw materials 
as well as unknown consumer demands.
Practically, excess inventory results in undesirable holding costs, 
while shortages result in disruptions 
in the production or supply chain and a potential loss of clients.
In view of this tradeoff, 
it is desirable to maintain inventory in all parts of the network 
within reasonable bounds in spite of the operating uncertainty.

In this context, questions of stability, reachability and robustness naturally arise.
Such questions have been previously considered in the control literature on
production-distribution or multi-storage systems where the unknown but bounded disturbances 
as well as the control inputs are assumed to take values in analog sets,
and the invariant sets of interest 
range from hyperboxes to polytopes, ellipsoids, and more generally, convex sets
\cite{JOUR:BlRiUk1997, JOUR:BlRiUk1997b, JOUR:BlMiUk2000, BOOKCHAPTER:BlMPRU2001, JOUR:BaGiPe2010}.
Related research directions in operations research address optimization 
problems in uncertain, multi-period inventory models.
For instance, base stock policies 
can be interpreted in terms of guaranteeing upper bounds on robustly control invariant sets.
Traditionally, cost functions penalizing holding and storage costs are used in conjunction with 
dynamic programming techniques to robustly maintain the stock within a neighborhood of zero.
In a bid to overcome the dimensionality problems inherent in dynamic programming 
and to mitigate undesirable dynamics,
robust optimization based approaches \cite{JOUR:BerThi2006} such as 
Affinely Adjustable Robust Counterpart (AARC) methods,
Globalized Robust Counterpart (GRC) methods \cite{JOUR:BeGoSh2009}, and various extensions,
have been developed.
In the supply chain literature, 
the network topology is often restricted to a tree, chain,
or other special structure that can be exploited in the solution.

Robustly control invariant sets have also been studied in a more general control context
due to their relevance to a variety of control problems, including constrained control and robust model predictive control.
Early results exploited dynamic programming techniques to characterize
invariant sets \cite{JOUR:BerRho1971}, 
and derived necessary and sufficient conditions for robust control invariance, 
expressed in terms of set equalities \cite{JOUR:Bertse1972}. 
An in-depth survey of  the existing body of literature in control, emphasizing techniques that exploit
connections with control Lyapunov functions, 
dynamic programing, and classical analytical results can be found in \cite{JOUR:Bl1999}. 
More recent characterizations of robustly control invariant sets for discrete-time linear 
systems were proposed in \cite{JOUR:RaKeMa2007}, 
allowing for tractable, convex-optimization based analysis approaches.

In this paper, we consider discrete-time dynamic network flow problems
in which the inputs are restricted to take their values in finite alphabets,
and we focus on a class of polytopic invariant sets (hyperboxes).
We seek to address the following two questions:
Under what conditions does a robustly control invariant hyperbox exist?
And under what conditions is such a set robustly globally attractive?
We present two main results:
First, we derive a necessary and sufficient condition for the existence of robustly control invariant hyperboxes. 
Second, we show that a stricter version of this condition is sufficient,
though not necessary in general, 
to guarantee robust global convergence of all state trajectories  to the invariant set. 
While the problems are posed as existence questions, 
our results touch on analysis and synthesis problems as well.
As the proposed proofs of sufficiency are constructive,
they allow one to immediately extract the corresponding robust control laws.
Moreover, they establish sufficient conditions for verifying invariance of certain sets.
Finally, we highlight connections between our derived results and existing ones
in the literature for the traditional setting.

The novelty in our setup,
which distinguishes it from most of the above referenced literature, 
is in the discrete nature of the control and disturbance inputs in conjunction with set based models of uncertainty.
Discrete action spaces are justifiable from both practical and theoretical perspectives:
From a practical standpoint, goods are usually processed, transported and distributed in batches. 
From a theoretical standpoint, 
the study of systems under discrete controls and disturbances has sparked much interest in recent years 
as evidenced by the literature on finite alphabet control 
\cite{JOUR:GooQue2003, JOUR:TaMeDa2011,JOUR:TaMeDa2008, JOUR:Tarraf2012}, 
mixed integer model predictive control \cite{AVH10}, 
discrete team theory \cite{JOUR:DeWVan2000} and boolean control \cite{B09}
where the problems of interest are often formulated as min-max games \cite{BOOK:BasOls1999}.
Note that while the setup in \cite{JOUR:Bertse1972} is general enough to allow discrete inputs,
our derivation differs significantly as do our proposed conditions,
being stated explicitly in terms of conditions on the alphabet sets rather than set equivalence conditions.

{\it Organization:} We present the problem statement in Section \ref{Ssec:Statement}
and explain its practical significance in Sections \ref{Ssec:RelevanceModel} and 
\ref{Ssec:RelevanceProblem}.
We state the main results in Section \ref{Sec:MainResults},
present complete proofs in Section \ref{Sec:Derivation},
discuss connections to the existing literature in Section \ref{Sec:Discussion},
present simple illustrative examples in Section \ref{Sec:IllustrativeExamples},
and conclude with directions for future work in Section \ref{Sec:Conclusions}.

{\it Notation:} $\mathbb{R}$, $\mathbb{Z}$, $\mathbb{R}_+$ and $\mathbb{Z}_+$ 
denote the reals, integers, non-negative reals and non-negative integers, respectively. 
$[x]_i$ denotes the $i^{th}$ component of $x \in \mathbb{R}^n$,
while $\mathbf{1}$ denotes the vector in $\mathbb{R}^n$ whose entries are all equal to one.
$hull\{S\}$ and $int(S)$ denote the convex hull
and interior, respectively, of $S \subset \mathbb{R}^n$.
$\mathbb{B}^n$ denotes the set of vertices of the unit hypercube,
that is $\mathbb{B}^n = \{0,1\}^n$. 
For $\mathcal{A} \subset \mathbb{Z}$, $M \in \mathbb{Z}^{n \times m}$,
$M \mathcal{A}^m$ denotes the image of $\mathcal{A}^m$ by $M$,
that is 
$M \mathcal{A}^m = \{ b \in \mathbb{Z}^n | b = M a \textrm{     for some    } a \in \mathcal{A}^m\}$.
Given a set $X = [0,x_1^+] \times \hdots \times [0,x_n^+] \subset \mathbb{R}^n$, 
$V_{X}$ denotes its set of vertices, that is
$V_X = \{ x \in \mathbb{R}^n | [x]_i \in \{0,x_i^+\} \}$.

\section{Problem Setup}
\label{Sec:Problem}

\subsection{Problem Statement}
\label{Ssec:Statement}

Consider the system described by
\begin{equation}
\label{eq:Dynamics}
x(t+1) = x(t) + B u(t) - D w(t) 
\end{equation}
where $t \in \mathbb{Z}_+$, 
state $x(t) \in \mathbb{R}^n$,
control input $u(t) \in \mathcal{U}^{m}$
and disturbance input $w(t) \in \mathcal{W}^p$.
The control and disturbance alphabet sets 
$\mathcal{U}= \{a_1, \hdots, a_r\} \subset \mathbb{Z}$ and
$\mathcal{W} = \{b_1, \hdots, b_q\} \subset \mathbb{Z}$,
respectively, 
are ordered with $a_1 < \hdots < a_r$ and $b_1 < \hdots < b_q$. 
$B \in \mathbb{Z}^{n \times m}$, $D \in \mathbb{Z}^{n \times p}$ are given.
This system is a fairly general model of a logistic network,
as we will describe in Section \ref{Ssec:RelevanceModel}.

\smallskip
\begin{definition}
\label{def:RCIS}
A hyperbox $X = [0,x_1^+] \times \hdots \times [0,x_n^+] \subset \mathbb{R}^n$ 
is robustly control invariant if there exists a control law
$\varphi: X \rightarrow \mathcal{U}^m$ such that for every 
$x(t) \in X$, $x(t+1) = x(t) + B \varphi(x(t)) - D w(t) \in X$ for any disturbance $w(t) \in \mathcal{W}^p$. 
\end{definition}
\smallskip

\smallskip
\begin{remark}
When $X = [0,x_1^+] \times \hdots \times [0,x_n^+]$ is robustly control invariant, 
then so is any other hyperbox
$X' = [x_1^-,x_1^-+x_1^+] \times \hdots \times [x_n^-,x_n^-+x_n^+]$.
Indeed, control law $\varphi' : X' \rightarrow \mathcal{U}^m$
defined by $\varphi' (x) = \varphi(x-x^-)$,
where $[x^-]_i = x_i^-$, verifies this assertion.
\end{remark}
\smallskip

\smallskip
\begin{definition}
\label{def:GA}
A hyperbox $X = [0,x_1^+] \times \hdots \times [0,x_n^+] \subset \mathbb{R}^n$ 
is robustly globally attractive if there exists a control law
$\psi:\mathbb{R}^n \setminus X \rightarrow \mathcal{U}^m$ such that for every
initial condition $x(0) \in \mathbb{R}^n \setminus X$ and 
disturbance $w: \mathbb{Z}_+ \rightarrow \mathcal{W}^p$,
the corresponding state trajectory
satisfies $x(\tau) \in X$ for some $\tau \in \mathbb{Z}_+$. 
\end{definition}
\smallskip

We are interested in answering two questions for the system described in (\ref{eq:Dynamics}):

\smallskip
{\bf Question 1:} Under what conditions does a robustly control invariant hyperbox exist? 
\smallskip 

{\bf Question 2:} Under what conditions does a robustly globally attractive and 
control invariant hyperbox exist?
\smallskip

\subsection{Relevance of the Model} 
\label{Ssec:RelevanceModel}

The dynamics in (\ref{eq:Dynamics}) constitute a unified and fairly general 
abstract model of logistic networks such as 
production networks, distribution networks and transportation networks.

For instance, a production system consists of products in one of several 
possible forms (raw products, intermediate products or finished products),
together with various  production processes 
by which a product is transformed, or several products are combined, 
to yield a new intermediate or final product. 
This production system can be modeled as a network 
in which the products are represented as nodes
while the production processes are represented by (weighted) hyper-arcs.
These production processes are typically individually controlled by the operator of the network:
The hyper-arcs are thus associated with control inputs.
Additionally, the network may interact with its external environment through
both controlled and uncontrolled flows representing generally uncertain
supplies of raw material as well as demands of various finished products.
In the mathematical model (\ref{eq:Dynamics}) of this production network, 
the state vector $x(t)$ represents the amounts of the various products 
available in the network at time $t$: Specifically, $[x(t)]_i$ represents the amount of 
product $i$ in the network at time $t$. 
The `$Bu$' term encodes the production processes and nominal supplies,
and the rate at which they proceed, controlled by the network operator.
The `$Dw$' term, in contrast, encodes uncertain supplies and demands,
or possibly uncertainty in the production processes.
Specifically, matrices $B$ and $D$ represent the network topology,
with non-zero entries indicating hyper-arcs and their weights
and/or locations of external flows.
Inputs $u$ and $w$ represent the controlled and uncontrolled flows, respectively.
The amount of product $i$ at the end of a production cycle thus equals the 
amount available at the beginning of the production cycle,
plus any amounts produced,
minus any amounts depleted during the production cycle.

\begin{figure} [htb]
\centering
\includegraphics[width=8cm]{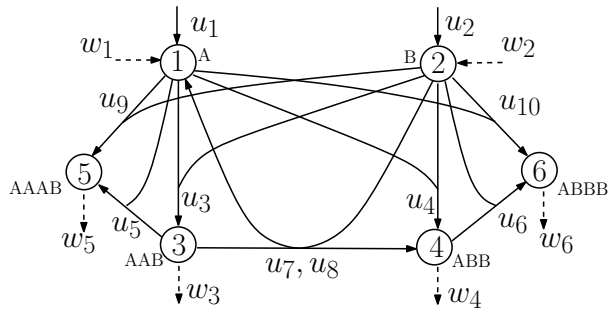}
\caption{Network representing a production system with 6 products}  
\label{fig:example3}
\end{figure}

To make things more concrete, 
consider a chemical production system in which 2 raw products,
referred to here as products A and B, can be used to produce 4 finished products,
say AAAB, AAB, ABB and ABBB.
This production system can be described by a network with 6 nodes,
each representing one of the 6 products, as shown in Figure \ref{fig:example3}.
Final product AAAB (represented by node 5 in the network) can be obtained 
through one of 2 production processes,
namely by combining products A and B in a 3:1 ratio, 
or by combining products A and AAB in a 1:1 ratio.
These two processes are represented by 2 hyper-arcs incoming into
node 5 in Figure 1, 
and are associated with control inputs $u_9$ and $u_5$, respectively,
as shown.
In contrast, product ABB (represented by node 4 in the network) 
can be produced by combining raw materials A and B in a 1:2 ratio,
or from product AAB by releasing one unit of A and adding one unit of B.
These two processes are represented by two hyper-arcs associated with 
control inputs $u_4$ and $u_8$, respectively.
Likewise, the two other finished products can be obtained by various production 
processes represented by hyper-arcs in the network,
and associated with control inputs $u_3$, $u_6$ and $u_7$.
The supply of raw products A and B is uncertain, 
represented by nominal flow rates associated with control inputs $u_1$ and $u_2$ 
and fluctuations in supply rates associated with disturbance inputs $w_1$ and $w_2$.
The demand for the 4 finished products is uncertain, 
represented by disturbance inputs $w_3$ through $w_6$.  
The $B$ matrix of the mathematical model describing this network is given by:
\begin{equation}
\label{Eq:MatrixB}
B=\left[\begin{array}{rrrrrrrrrr} 
1 & 0 & -2 & -1 & -1 & 0 & -1 & 1 & -3 & -1\\ 
0 & 1 & -1 & -2 & 0 & -1 & 1 & -1 & -1 & -3\\
0 & 0 & 1 & 0 & -1 & 0 & 1 & -1 & 0 & 0\\
0 & 0 & 0 & 1 & 0 & -1 & -1 & 1 & 0 & 0\\
0 & 0 & 0 & 0 & 1 & 0 & 0 & 0 & 1 & 0\\
0 & 0 & 0 & 0 & 0 & 1 & 0 & 0 & 0 & 1\\
\end{array}\right].
\end{equation}
In this representation, the $i^{th}$ column of the $B$ matrix represents the hyper-arc associated
with control input $u_i$: 
For instance,
the hyper-arc associated with input $u_9$ is represented by column 9.
The non-zero entries of the column indicate that three units of product A (represented by node 1) 
and one unit of product B (represented by node 2) and used to produce one unit of 
product AAAB (represented by node 5). 
As such, matrix $B$ encodes the topology of the network, 
namely the hyper-arcs and their weights.
Similarly, the $D$ matrix is given by:
\begin{equation}
\label{Eq:MatrixD}
D=\left[\begin{array}{rrrrrr} 
-1 & 0 & 0 & 0 & 0 & 0 \\ 
0 & -1 & 0 & 0 & 0 & 0 \\ 
0 & 0 & 1 & 0 & 0 & 0 \\ 
0 & 0 & 0 & 1 & 0 & 0 \\ 
0 & 0 & 0 & 0 & 1 & 0 \\ 
0 & 0 & 0 & 0 & 0 & 1 \\ 
\end{array}\right].
\end{equation}

Likewise in the case of distribution and transportation networks, 
the nodes of the network represent warehouses and transportation hubs, respectively, 
with the $i^{th}$ component of the state vector thus representing the quantity of commodities 
present at the $i^{th}$ warehouse/hub.
The `$Bu-Dw$' term encodes the various transportation routes,
distribution protocols, supplies and demands, 
with matrices $B$, $D$ representing the network topology
and inputs $u$ and $w$ again respectively representing the controlled flows and uncertainty in the system.

\subsection{Relevance of the Problem Statement} 
\label{Ssec:RelevanceProblem}

In this setting, it is intuitively desirable to contain each component of the state vector 
within two bounds, a zero lower bound and a positive upper bound.
The zero lower bound guards against shortages and interruptions in the production process,
or against the underuse of distribution and transportation resources.
The upper bound ensures that the storage capabilities of the system are not exceeded.
The question of existence of robustly control invariant sets, 
specifically hyperboxes (Question 1 in Section \ref{Ssec:Statement}), thus naturally arises.

Moreover in this setting, 
the model of uncertainty (specifically the choice of set $\mathcal{W}$) 
encodes the typical uncertainty encountered in day to day operations.
Since it is impossible to rule out rare occurrences of large unmodeled uncertainty 
that would drive the system away from its typical operation,
such as emergencies or catastrophic events,
it is reasonable to question whether the system can recover from such events:
The question of robust global attractivity of the robustly control invariant hyperboxes 
(Question 2 in Section \ref{Ssec:Statement}) thus also naturally arises.
 
\section{Main Results}
\label{Sec:MainResults}

Consider the following sets defined for $i \in \{1,\hdots,n \}$:
\begin{eqnarray*}\label{sets}
\nonumber \mathcal{U}_+^i & = & \{ u \in \mathcal{U}^m | [B u - D w]_i \geq 0, \forall w \in \mathcal{W}^p \}, \\
\nonumber \mathcal{U}_-^i & = & \{ u \in \mathcal{U}^m | [B u - D w]_i \leq 0, \forall w \in \mathcal{W}^p \}, \\
\nonumber \mathcal{U}_{+*} ^{i}& = & \{ u \in \mathcal{U}^m | [B u - D w]_i  > 0, \forall w \in \mathcal{W}^p \}, \\
\nonumber \mathcal{U}_{-*}^{i} & = & \{ u \in \mathcal{U}^m | [B u - D w]_i < 0, \forall w \in \mathcal{W}^p \}. 
\end{eqnarray*}
Associate with every $x \in \mathbb{R}_+^n$ a {\it signature},
namely an n-tuple $(s_1,\hdots,s_n)$ with $s_i=+$ if $[x]_i=0$ and $s_i=-$ if $[x]_i >0$,
and two subsets of $\mathcal{U}^m$ defined by
\begin{eqnarray*}
\mathcal{U}_x = \mathcal{U}_{s_1}^1 \cap \hdots \cap \mathcal{U}_{s_n}^n, \\
\mathcal{U}_x^* =  \mathcal{U}_{s_1*}^{1} \cap \hdots \cap \mathcal{U}_{s_n*}^{n}.
\end{eqnarray*}

We can now state the main results. 
The first provides a complete answer to Question 1: 

\begin{theorem}
\label{Thm:existence}
The following two statements are equivalent:
\begin{enumerate}
\item [(a)] There exists a set $X = [0,x_1^+] \times \hdots \times [0,x_n^+] $ that is robustly control invariant.
\item [(b)] The following condition holds
\begin{equation}
\label{Eq:SCexistence}
\mathcal{U}_z \neq \emptyset, \textrm{     } \forall z \in \mathbb{B}^n.
\end{equation}
\end{enumerate}
\end{theorem}

The second proposes a sufficient condition for the existence of 
a robustly control invariant set that is globally attractive, 
thus giving a partial answer to Question 2. 
In general, this condition need not be necessary as discussed in Section \ref{Sec:Derivation}.

\begin{theorem}
\label{Thm:ConvergenceSufficiency}
If the following condition holds 
\begin{equation}
\label{Eq:SCconvergence}
\mathcal{U}_z^*  \neq \emptyset, \textrm{      } \forall z \in \mathbb{B}^n
\end{equation}
there exists a robustly control invariant and globally attractive set 
$X=[0, x_1^+] \times \hdots \times [0, x_n^+]$.
\end{theorem}

\begin{remark}
While it is somewhat disappointing that both conditions 
$(\ref{Eq:SCexistence})$ and $(\ref{Eq:SCconvergence})$ are combinatorial in nature,
and thus grow exponentially with the number of nodes in the network,
the result is not surprising in view of the established results for the case of analog alphabets.
Indeed in \cite{JOUR:BlRiUk1997, BOOKCHAPTER:BlMPRU2001}, 
discrete-time networks subject to analog control inputs and disturbances were analyzed. 
For instance, it was shown that when
$\mathcal{W} = [d^-,d^+] $, $\mathcal{U} = [u^-,u^+]$ 
for some scalars $d^- < d^+$, $u^- < u^+$,
the condition 
\begin{displaymath}
D\mathcal{W}^p \subset int(B \mathcal{U}^m)
\end{displaymath}
is necessary and sufficient for the existence of a globally attractive robustly control invariant set.
While this condition appears deceptively simple,
verifying it is known to be NP-hard:
Indeed, it requires checking $2^n-2$ constraints \cite{JOUR:BlRiUk1997} in general.
\end{remark}

\section{Derivation of the Main Results}
\label{Sec:Derivation}

\subsection{Existence of a Robustly Control Invariant Set}
\label{SSec:Existence}

We begin by establishing a necessary condition for a given control law to 
render a set robustly control invariant.

\begin{lemma}
\label{Lemma:NCexistence}
Let $X = [0,x_1^+] \times \hdots \times [0,x_n^+]$ be robustly control invariant,
and consider a control law
$\varphi:X \rightarrow \mathcal{U}^m$ as in Definition \ref{def:RCIS}.
Then at every vertex $x \in V_{X}$, we have
\begin{displaymath}
\varphi(x) \in \mathcal{U}_{z}
\end{displaymath}
where $z$ is the unique element of $\mathbb{B}^n$ whose signature is identical to that of $x$.
\end{lemma}

\begin{proof}
Assume that $\varphi(x) \notin \mathcal{U}_{z}$ 
for some $x \in V_X$ with signature $(s_1,\hdots,s_n)$,
and consider a set $X = [0, x_1^+] \times \hdots \times [0,x_n^+]$, $x_i^+ >0$.
Pick a $u \in \mathcal{U}^m \setminus \mathcal{U}_z$.
By assumption, there exists an $i \in \{1,\hdots,n\}$ such that $u \notin \mathcal{U}_{s_i}^i$.
Now letting $x(t)=x$ and applying control input $\varphi(x(t))=u$, we have 
$ [x(t+1)]_i =[x(t) + Bu(t)-Dw(t)]_i = [x]_i +[Bu-Dw]_i$,
and $[x(t+1)]_i - [x(t) ]_i = [Bu(t)-Dw(t)]_i$ satisfies
\begin{displaymath}
[x(t+1)]_i - [x(t)]_i  >0 
\end{displaymath}
for some $w \in \mathcal{W}^p$ when $[x]_i \neq 0$
and
\begin{displaymath}
[x(t+1)]_i - [x(t)]_i  <0 
\end{displaymath}
for some $w \in \mathcal{W}^p$ when $[x]_i = 0$.
Hence $x(t+1) \notin X$ for some $w \in \mathcal{W}^p$.
Noting that the choice of $u \in \mathcal{U}^m \setminus \mathcal{U}_z$ 
was arbitrary allows us to conclude that $X$ is not robustly control invariant.
\end{proof}

\begin{lemma}
\label{Lemma:SCexistence}
If the following condition holds  
$$
\mathcal{U}_z  \neq \emptyset, \textrm{     } \forall z \in \mathbb{B}^n
\reqno{Eq:SCexistence}
$$
the set $X= [0,2L_1] \times \hdots \times [0,2L_n]$ 
is robustly control invariant for any choice $u_z \in \mathcal{U}_z$
whenever $\displaystyle L_i \geq L_i^o = \max_{w,z} \Big| [Bu_z -Dw]_i \Big|$.
\end{lemma}

\begin{proof}
Assume that $\mathcal{U}_z \neq \emptyset$ holds for all $z \in \mathbb{B}^n$
and pick for each $z \in \mathbb{B}^n$ a control input $u_z \in \mathcal{U}_z$.
The set $X = [0, 2 L_1] \times \hdots \times [0, 2 L_n]$ is robustly control invariant.
Indeed, consider the control law $\varphi: X \rightarrow \mathcal{U}^m$ defined by
$\varphi(x)=u_{z(x)}$, where $z(x) \in \mathbb{B}^n$ is the unique vertex of  the unit hypercube with signature
$s_i=+$ if $[x]_i \leq L_i$ and $s_i=-$ otherwise.
Note that under this control law, we have for every $i \in \{1,\hdots,n \}$
\begin{eqnarray*}
[x(t+1)]_ i = [x(t)]_i + [B \varphi(x(t))- D w(t)]_i
\end{eqnarray*}
Thus by construction, when $0 \leq [x(t)]_i \leq L_i$,
$0 \leq [B \varphi(x(t)) - D w(t)]_i  \leq L_{i}^o \leq L_i$
and $0 \leq [x(t+1)]_i \leq 2 L_i$.
Likewise when $L_i < [x(t)]_i \leq 2L_i$,
$-L_i \leq -L_i^o \leq [B \varphi(x(t)) - D w(t)]_i  \leq 0$
and $0 \leq [x(t+1)]_i \leq 2 L_i$.
It follows that $X$ is robustly control invariant.
\end{proof}

\begin{remark}
Note that when the sufficient condition (\ref{Eq:SCexistence}) holds,
our proof effectively provides an immediate construction of a 
full state feedback control law rendering any sufficiently large hyper box robustly control invariant.
For any choice of $u_z \in \mathcal{U}_z$, $z \in \mathbb{B}^n$,
the corresponding robustly control invariant set  $X= [0,2L_1^o] \times \hdots \times [0,2L_n^o]$
can be interpreted as an outer set (or superset) of the smallest robustly control invariant set.
\end{remark}

We are now ready to prove Theorem \ref{Thm:existence}:

{\it Proof of Theorem \ref{Thm:existence}:}
(a) $\Rightarrow$ (b): 
Assume that $\mathcal{U}_z= \emptyset$ for some $z \in \mathbb{B}^n$ with signature 
$(s_1,\hdots,s_n)$,
consider a set $X = [0, x_1^+] \times \hdots \times [0,x_n^+]$, $x_i^+ >0$,
and pick vertex $x \in V_x$ with signature $(s_1,\hdots,s_n)$.
For any control law $\varphi: X \rightarrow \mathcal{U}^m$, 
we necessarily have $\varphi(x) \notin \mathcal{U}_z$ (as the latter set is empty).
Noting that the choice of control law was arbitrary allows us to conclude that $X$ 
is not robustly control invariant.
Noting that the choice of $X$ was arbitrary allows us to conclude that a robustly control invariant set cannot exist,
thus completing our argument.

(b) $\Rightarrow$ (a): Follows directly from Lemma \ref{Lemma:SCexistence} 
which provides an explicit construction for such a set.
{\hspace{\stretch{1}}} $\blacksquare$

\subsection{Robust Global Attractivity}
\label{SSec:Attractivity}

We begin by establishing sufficient conditions for a set to be robustly globally attractive.
Our proof hinges on the construction of an appropriate Lyapunov-like function.
 
\begin{lemma}
\label{Lemma:ConvergenceSet}
If the following condition holds 
$$
\mathcal{U}_z^*  \neq \emptyset, \textrm{      } z \in \mathbb{B}^n
\reqno{Eq:SCconvergence}
$$
set $X= [0,2L_1^*] \times \hdots \times [0,2L_n^*]$ with
$\displaystyle L_i^* = \max_{w,z} \Big| [Bu_z -Dw]_i \Big|$
is robustly globally attractive for any choice $u_z \in \mathcal{U}_z^*$.
\end{lemma}

\begin{proof}
Pick a choice $u_z \in \mathcal{U}_z^*$, for $z \in \mathbb{B}^n$.
To prove global attractiveness of the corresponding set $X$, 
consider the control law $\zeta: \mathbb{R}^n \setminus X \rightarrow \mathcal{U}^m$ defined by
$\zeta(x)=u_{z(x)}$, where $z(x) \in \mathbb{B}^n$ is the unique vertex of  the unit hypercube with signature
$s_i=+$ if $[x]_i \leq L_i$ and $s_i=-$ otherwise.
Letting
$ \displaystyle \Delta = \min_{i,z,w} \Big| [Bu_z - D w]_i \Big|$, 
note that $\Delta >0$.
Consider the function $V: \mathbb{R}^n \rightarrow \mathbb{R}$ defined by
\begin{displaymath}
V(x) = \max_i \min_{y \in X} \Big| [x]_i - [y]_i \Big|.
\end{displaymath}
Note that by construction, we have
$V(x) \geq 0$, for all $x \in \mathbb{R}^n$ and $V(x) = 0 \Leftrightarrow x \in X$.
Moreover, observe that under control law $\zeta$ we have 
\begin{equation}
\label{Eq:Lyapunov}
V(x(t+1)) - V (x(t)) <0
\end{equation}
along system trajectories whenever $x(t) \in \mathbb{R}^n \setminus X$.
To verify this, consider any $x(t) \in \mathbb{R}^n \setminus X$,
and consider the $i^{th}$ coordinate direction.
We have:
\begin{equation}
\label{Eq:Alongi}
\min_{y \in X} \Big| [x(t+1)]_i - [y]_i \Big|  \leq \min_{y \in X} \Big| [x(t)]_i - [y]_i \Big|
\end{equation}
with equality holding {\it only} in the case where the right hand side is {\it zero}.
Indeed, we can distinguish three cases:
\begin{itemize}
\item $[x(t)]_i > 2 L_i^*$: In this case, by construction we have $0 < [x(t+1)]_i < [x(t)]_i$,
and
\begin{eqnarray*}
\max\{0, [x(t+1)]_i - 2L_i^* \} & = & \min_{y \in X} \Big| [x(t+1)]_i - [y]_i \Big| \\
& < & \min_{y \in X} \Big| [x(t)]_i - [y]_i \Big| \\
& = & [x(t)]_i - 2L_i^*
\end{eqnarray*}

\item $0 \leq [x(t)]_i \leq 2 L_i^*$: In this case, by construction we have $0 \leq [x(t+1)]_i \leq 2 L_i^*$,
and
\begin{displaymath}
\min_{y \in X} \Big| [x(t+1)]_i - [y]_i \Big| = \min_{y \in X} \Big| [x(t)]_i - [y]_i \Big| = 0
\end{displaymath}

\item $[x(t)]_i < 0$: In this case, by construction we have $[x(t)]_i < [x(t+1)]_i < 2L_i^*$,
and
\begin{eqnarray*}
\max\{0, - [x(t+1)]_i  \} & = & \min_{y \in X} \Big| [x(t+1)]_i - [y]_i \Big| \\
& < & \min_{y \in X} \Big| [x(t)]_i - [y]_i \Big| \\
& = & - [x(t)]_i 
\end{eqnarray*}
\end{itemize}
Equation (\ref{Eq:Lyapunov}) thus follows from the the fact that (\ref{Eq:Alongi}) holds for each $i \in \{1,\hdots,n\}$,
with equality only when both terms are identically 0.

Moreover, we have 
\begin{equation}
\label{Eq:BoundLyapunov}
V(x(t+1)) \leq V(x(t)) - \Delta
\end{equation}
whenever $x(t) \in \mathbb{R}^n \setminus X$ and $x(t+1) \in \mathbb{R}^n \setminus X$.

Thus, for any choice of initial condition $x(0)$ and of disturbance input $w : \mathbb{Z}_+ \rightarrow \mathcal{W}^p$, 
we conclude from (\ref{Eq:Lyapunov}) that $\displaystyle \lim_{t \rightarrow \infty}V(x(t)) \rightarrow 0$.
Moreover, we conclude from (\ref{Eq:BoundLyapunov}) that
there must exist a $\tau > 0$ such that $V(x(\tau)) =0$, or equivalently $x(\tau) \in X$. 
Hence $X$ is globally attractive. 
\end{proof}

We are now ready to prove Theorem \ref{Thm:ConvergenceSufficiency}:

{\it Proof of Theorem \ref{Thm:ConvergenceSufficiency}:}
The proof is constructive. Assume that (\ref{Eq:SCconvergence}) holds 
and pick a choice $u_z \in \mathcal{U}_z^*$, for $z \in \mathbb{B}^n$.
The set $X = [0, 2 L_1^*] \times \hdots \times [0, 2 L_n^*]$ is robustly control invariant
by Lemma \ref{Lemma:SCexistence}, since $\mathcal{U}_{z}^* \subseteq \mathcal{U}_z$.
Moreover, $X$ is globally attractive by Lemma \ref{Lemma:ConvergenceSet}.
{\hspace{\stretch{1}}} $\blacksquare$

In the case of a degenerate network consisting of a single node
(i.e. when $n=1$, $p$ and $m$ arbitrary),
condition (\ref{Eq:SCconvergence}) is necessary as well as sufficient for robust global attractivity:

\begin{proposition}
\label{Proposition:NSCscalar}
When $n=1$, if $\mathcal{U}_z^* = \emptyset$ for some $z\in \{0,1\}$,
there cannot exist a set $X=[0,L]$ that is robustly globally attractive.
\end{proposition}

\begin{proof}
Assume, without loss of generality, that $\mathcal{U}_1^* = \emptyset$
and consider a set $X=[0,L]$.
It follows that for any $x(t) >L$ and any choice of control law, 
$x(t+1) \geq x(t)$ for some $w(t) \in \mathcal{W}^p$,
call this $w_o$.
We can thus always find a disturbance input, 
namely $w(t)=w_o$, $t \geq 0$, for which any state trajectory initialized to the right of the interval $[0,L]$
will remain to the right of the interval for all times,
and $X$ is not robustly control invariant.
\end{proof}

In general, however, 
condition (\ref{Eq:SCconvergence}) is {\it not necessary} for global attractivity.
Indeed, consider the second order counterexample constructed as follows:

\begin{counterex}
\label{CounterEx:NonNecessaryConvergence}
Let $n=2$, $m=2$, $p=1$,
$\mathcal{W}=\{ 0 \}$, $\mathcal{U}=\{-1,3\}$, and
$\displaystyle B = \left[ \begin{array}{cc}
1 & 3 \\
1 & -1
\end{array} \right]$.
In this case, we have four possible control inputs,
\begin{eqnarray*}
\mathcal{U}^2 & = & \Big\{
u_1 = \left[ \begin{array}{c}
-1 \\ -1
\end{array} \right],
u_2=\left[ \begin{array}{c}
-1 \\ 3
\end{array} \right],
u_3= \left[ \begin{array}{c}
3 \\ -1
\end{array} \right],
u_4=\left[ \begin{array}{c}
3 \\ 3
\end{array} \right]
\Big\}
\end{eqnarray*}
Computing the relevant sets, we get
\begin{eqnarray*}
\begin{array}{l l l l}
\mathcal{U}_-^1= \{u_1,u_3\}, \quad \quad & 
\mathcal{U}_{-*}^1= \{u_1\}, \quad \quad &
\mathcal{U}_+^1 = \{ u_2,u_3,u_4\}, \quad \quad & 
\mathcal{U}_{+*}^1= \{u_2,u_4\},\\
\mathcal{U}_-^2= \{u_1,u_2,u_4\}, & 
\mathcal{U}_{-*}^2 = \{u_2\}, &
\mathcal{U}_+^2 = \{u_1,u_3,u_4 \}, & 
\mathcal{U}_{+*}^2 =\{ u_3 \}.\\
\end{array}
\end{eqnarray*}
On the unit hypercube, we have 
\begin{eqnarray*}
\begin{array}{ll}
\mathcal{U}_{[0,0]'} = \{u_3,u_4\}, \quad \quad \quad & 
\mathcal{U}_{[0,0]'}^* = \emptyset, \\
\mathcal{U}_{[0,1]'} = \{u_2, u_4\}, \quad \quad \quad & 
\mathcal{U}_{[0,1]'}^* = \{u_2\}, \\
\mathcal{U}_{[1,0]'} = \{u_1,u_3\}, \quad \quad \quad & 
\mathcal{U}_{[1,0]'}^* = \emptyset, \\
\mathcal{U}_{[1,1]'} = \{u_1\}, \quad \quad \quad & 
\mathcal{U}_{[1,1]'}^* = \emptyset. \\
\end{array}
\end{eqnarray*}
By Lemma \ref{Lemma:SCexistence}, 
there exists a set that is robustly control invariant, namely
$X=[0,24] \times [0,8]$.
While condition (\ref{Eq:SCconvergence}) does not hold, 
it is easy to note that $X$ is also robustly globally attractive.
Indeed, consider the control law $\varphi: \mathbb{R}^2 \rightarrow \mathcal{U}^2$ defined by
\begin{displaymath}
\varphi(x) =
\left\{ \begin{array}{cc}
u_1 & \textrm{     when     } [x]_1 \geq 12, 0 \leq [x]_2 \leq 8 \\
u_4 & \textrm{     when     } [x]_1 < 12, 0 \leq [x]_2 \leq 8 \\
u_3 & \textrm{     when     } [x]_2 < 0 \\
u_2 & \textrm{     when     } [x]_2 > 8 \\
\end{array} \right.
\end{displaymath}
It is straightforward to verify that $\varphi$ renders $X$ 
robustly control invariant and globally attractive.
\end{counterex}

\section{Discussion}
\label{Sec:Discussion}

In this Section, we establish connections between our results and existing results,
focusing in particular on certain set inclusion conditions 
that often appear in the literature on set invariance in control.

\subsection{Connections to Set Inclusion Conditions on the Alphabet Sets}

The necessary and sufficient condition for existence 
of a robustly control invariant set,
as well as the sufficient condition for global attractivity 
are both formulated in terms of combinatorial conditions. 
In contrast,
conditions derived in the literature for the existence of robustly control invariant sets for 
discrete-time dynamic network flow models with analog-valued inputs are typically 
presented as set inclusion conditions. 
Specifically in \cite{JOUR:BlRiUk1997}
the authors prove that 
\begin{equation}
\label{Eq:AnalogNSC}
B \mathcal U^m \supseteq D \mathcal W^p
\end{equation}
is necessary and sufficient for a robustly control invariant set to exist. 
In light of this, in this section we attempt to connect our combinatorial conditions with
appropriate set inclusion conditions.
In particular, we show that (\ref{Eq:SCexistence}) implies another condition, 
formulated in terms of the convex hull of $B \mathcal U^m$ and $D \mathcal W^p$. 
Likewise, we show that (\ref{Eq:SCconvergence}) implies another condition, 
formulated in terms of the interior of the convex hull of $B \mathcal U^m$ 
and the convex hull of $D \mathcal W^p$. 

However, it should be emphasized that this exercise is mainly academic 
for two reasons:
First, set inclusion conditions (\ref{Eq:AnalogNSC}) are known to be NP-hard to verify in general.
Indeed, verifying the condition
when $\mathcal{U}$ and $\mathcal{W}$ are closed convex sets 
requires checking $2^n-2$ constraints in general,
where $n$ is the dimension of the underlying state-space.
This characterization, which relies on an earlier result in \cite{JOUR:McCorm1996},
is stated and proved in Section 4 of \cite{JOUR:BlRiUk1997}.
As such, set inclusion conditions do not offer much promise of a substantial reduction in computational burden.
Second, the directions of the implications are such that a set inclusion
condition can only be used to conclude the non-existence of a robustly control invariant set
in the case where the condition is violated.

\vspace*{10pt}
\begin{lemma}
\label{Lemma:EquivalentSCexistence}
Condition (\ref{Eq:SCexistence}) holds iff
for each closed orthant $\overline{O}^j$, $j=1,\hdots,2^n$,
there exists a control input $u^j \in \mathcal{U}^m$ such that the set
\begin{displaymath}
Bu^j - D\mathcal{W}^p = \{x \in \mathbb{R}^n | x = Bu^j - Dw \textrm{   for some   } w \in \mathcal{W}^p\}
\end{displaymath}
satisfies $Bu^j - D \mathcal{W}^p \subset \overline{O}^j$.
\end{lemma}

\begin{proof}
Follows from the definitions by noting that each vertex $z$ of $\mathbb{B}^n$ 
can be uniquely associated with an orthant in $\mathbb{R}^n$,
namely the unique orthant containing $1/2 \cdot \mathbf{1} - z$.
\end{proof}

\vspace*{10pt}
\begin{lemma}
\label{Lemma:EquivalentSCconvergence}
Condition (\ref{Eq:SCconvergence}) holds iff
for each open orthant $O^j$, $j=1,\hdots,2^n$,
there exists a control input $u^j \in \mathcal{U}^m$ such that the set
\begin{displaymath}
Bu^j - D\mathcal{W}^p = \{x \in \mathbb{R}^n | x = Bu^j - Dw \textrm{   for some   } w \in \mathcal{W}^p\}
\end{displaymath}
satisfies $Bu^j - D \mathcal{W}^p \subset O^j$.
\end{lemma}

\begin{proof}
Follows from the definitions by noting that each vertex $z$ of $\mathbb{B}^n$ 
can be uniquely associated with an orthant in $\mathbb{R}^n$,
namely the unique orthant containing $1/2 \cdot \mathbf{1} - z$.
\end{proof}

\vspace*{10pt}
\begin{theorem}
Condition (\ref{Eq:SCexistence})
implies
\begin{equation}
\label{hull}  
hull\{B \mathcal U^m\} \supseteq hull \{D \mathcal W^p\}.
\end{equation}
The converse statement also holds when $n=1$.
\end{theorem}

\begin{proof}
When (\ref{Eq:SCexistence}) holds, by Lemma \ref{Lemma:EquivalentSCexistence} we have that for each closed orthant $\overline{O}^j$, there exists a $u^j \in \mathcal{U}^m$
such that $B u^j - D \mathcal{W}^p \subset \overline{O}^j$.
Let $\mathbf{U} = \{u^1,\hdots, u^{2^n} \}$.
Pick any $w \in \mathcal{W}^p$. We have (with some abuse of notation)
\begin{eqnarray*}
& B u^j - D \mathcal{W}^p \subset \overline{O}^j \textrm{    for   } j=1,\hdots, 2^n \\
\Rightarrow & 0 \in hull \{B u^j - Dw | u^j \in \mathbf{U} \} \\
\Leftrightarrow & 0 \in hull \{ B \mathbf{U} - Dw \} \\
\Rightarrow & 0 \in hull \{ B \mathcal{U}^m - Dw \} \\
\Leftrightarrow & 0 \in hull \{ B \mathcal{U}^m \} - Dw \\
\Leftrightarrow & \exists x \in hull\{ B \mathcal{U}^m \} \textrm{   such that   } x=Dw
\end{eqnarray*}
Since the choice of $w$ was arbitrary in $\mathcal{W}^p$, we conclude that for any $w \in \mathcal{W}^p$ there exists $x \in hull \{ B \mathcal{U}^m \}$ such that $x =Dw$.
Hence $D \mathcal{W}^p \subset hull \{ B \mathcal{U}^m \}$, and 
$hull \{ D \mathcal{W}^p \} \subseteq hull \{ B \mathcal{U}^m \}$.

When $n=1$, both implications in the above derivation become equivalences
(the second by picking $\mathbf{U}$ to correspond to extremal points),
and the converse statement holds.
\end{proof}

\vspace*{10pt}
\begin{theorem}
\label{leminthull}
Condition (\ref{Eq:SCconvergence})
implies 
\begin{equation}
\label{inthull} 
int ( hull\{B \mathcal U^m\}) \supset hull \{D \mathcal W^p\}.
\end{equation}
The converse statement also holds when $n=1$.
\end{theorem}

\begin{proof}
When (\ref{Eq:SCconvergence}) holds, 
by Lemma \ref{Lemma:EquivalentSCconvergence} we have that for each open 
orthant $O^j$, there exists a $u^j \in \mathcal{U}^m$
such that $B u^j - D \mathcal{W}^p \subset O^j$.
Let $\mathbf{U} = \{u^1,\hdots, u^{2^n} \}$.
Pick any $w \in \mathcal{W}^p$. We have (again with some abuse of notation)
\begin{eqnarray*}
& B u^j - D \mathcal{W}^p \subset O^j \textrm{    for   } j=1,\hdots, 2^n \\
\Rightarrow & 0 \in int( hull \{ B u^j - Dw | u^j \in \mathbf{U} \}) \\
\Leftrightarrow & 0 \in int(hull \{ B \mathbf{U} - Dw \}) \\
\Rightarrow & 0 \in int( hull \{ B \mathcal{U}^m - Dw \}) \\
\Leftrightarrow & 0 \in int (hull \{ B \mathcal{U}^m \} - Dw) \\
\Leftrightarrow & \exists x \in int( hull\{ B \mathcal{U}^m \}) \textrm{   such that   } x=Dw
\end{eqnarray*}
Since the choice of $w$ was arbitrary in $\mathcal{W}^p$, we conclude that for any $w \in \mathcal{W}^p$ there exists $x \in int (hull \{ B \mathcal{U}^m \})$ such that $x =Dw$.
Hence $D \mathcal{W}^p \subset int( hull \{ B \mathcal{U}^m \})$, and 
$hull \{ D \mathcal{W}^p \} \subseteq int (hull \{ B \mathcal{U}^m \})$.

When $n=1$, both implications in the above derivation become equivalences
(the second by picking $\mathbf{U}$ to correspond to extremal points),
and the converse statement holds.
\end{proof}

\subsection{Connections to Sub-Tangentiality Conditions}

Lemma \ref{Lemma:NCexistence} can be interpreted as a counterpart to the necessary 
condition in Nagumo's Theorem \cite{JOUR:Nagumo1942},
adapted to the discrete-time, forced, and discrete alphabet setting of interest here.
Nagumo's sub-tangentiality condition
and related conditions (see \cite{JOUR:Bl1999} for an overview),
while known to be sufficient for continuous-time systems
and linear discrete-time systems under analog inputs,
are not sufficient for general discrete-time systems.
As such, it is not surprising that additional constraints
need to be placed on the set to ensure sufficiency of condition (\ref{Eq:SCexistence})
in establishing set invariance.

\section{Illustrative Examples}
\label{Sec:IllustrativeExamples}

We begin with a simple scalar example for intuition.
We then revisit the production network
introduced in Section \ref{Ssec:RelevanceModel}.

\subsection{A Scalar Example}
\label{Ex:One}
Consider the scalar dynamics ($n=m=p=1$) given by
\begin{displaymath}
x(t+1) = x(t) + B u(t) - D w(t)
\end{displaymath}
with alphabets $\mathcal{U} =\{-100, -2, 3, 150\}$ and $\mathcal{W}= \{-6,4\}$.
We begin by computing sets $\mathcal{U}_+$ and $\mathcal{U}_-$ (no need for indices `$i$' in this case)
by inspecting Table \ref{Table:Ex1}
whose entries are simply the values of `$Bu-Dw$': 
This table can be interpreted as the payoff matrix of a zero sum game between players 
$u$ and $w$.

\begin{table}
\caption{Payoff matrix for zero sum game in Example 1}
\label{Table:Ex1}
\centering
\begin{tabular}{c || c | c | c}
$u / w$ & $-6$ & $4$ \\
\hline \hline
$-100$ &  $-100B+6D$ & $-100B-4D$\\
\hline
$-2$ &  $-2B+6D$ & $-2B-4D$ \\
\hline
$3$ & $3B+6D$ & $3B-4D$\\
\hline
$150$ & $150B+6D$ & $150B-4D$
\end{tabular}
\end{table}

We have $\mathcal{U}_- \neq \emptyset$ and $\mathcal{U}_+ \neq \emptyset$ iff
\begin{displaymath}
\left\{ \begin{array}{c}
-100B+6D \leq 0 \\
-100B - 4D \leq 0
\end{array} \right.
\textrm{       and        }
\left\{ \begin{array}{c}
150B+6D \geq 0 \\
150B - 4D \geq 0
\end{array} \right.
.
\end{displaymath}
We thus conclude that a robustly control invariant set indeed exists iff
$B \geq \max \{0.06D,-0.04D\}$,
and is moreover globally attractive provided strict equality holds.
Consider for example the case where $B=D=1$,
for which a robustly control invariant set is guaranteed to exist:
It is straighforward to verify that $X =[0, 157]$ is robustly control invariant
(it is in fact the smallest such set).

\subsection{A Six Node, Ten Arc Production Network}
\label{EX:Two}

We revisit the production process described in Section \ref{Ssec:RelevanceModel}
and represented by the network in Figure 1. 
We have $x(t) \in \mathbb{R}^6$,
$u(t) \in \mathcal{U}^{10}$,
$w(t) \in \mathcal{W}^6$,
and matrices $B$ and $D$ are given in (\ref{Eq:MatrixB}) and (\ref{Eq:MatrixD}),
respectively.
In this example, we assume that
$ \mathcal U = \{a \in \mathbb{Z} | 0 \leq a \leq 400 \}$,
representing the operator's ability to completely shut down a production process
or determine its rate up to some maximum level.
We assume that $\mathcal W = \{b \in \mathbb{Z} | 20 \leq b \leq 40 \}$.

Under these assumptions, condition (\ref{Eq:SCconvergence}) holds.
We proceed to verify this without explicitly constructing the sets $\mathcal{U}_z^*$,
by employing a heuristic approach.
Indeed, 
we solve the following linear program with decision variables $\lambda$ and $u$ for each vertex $z \in \mathbb{B}^6$: 
\begin{displaymath}
\begin{array}{rlc} 
\textrm{min   }& \lambda \\ 
\textrm{subject to    } & [Bu]_i  - 40 \geq \lambda & \mbox{if $[z]_i=0$}\\
& [Bu]_i - 20 \leq - \lambda & \mbox{if $[z]_i=1$}\\
& [u]_i \geq 0 \\
& [u]_i \leq 400 \\
& \lambda \geq  \epsilon  
\end{array}
\end{displaymath}
where $\epsilon > 0$ is a chosen parameter. 
Note that the linear program returns a fractional solution that can be rounded to an integer
solution within the admissible input set.
Also note that the existence, for each choice of $z \in \mathbb{B}^6$, 
of an integer solution satisfying the first four
LP constraints (for any $\lambda \geq 0$) effectively ensures satisfaction of (\ref{Eq:SCconvergence}),
since these constraints represent extremal (worst case) values of the disturbance inputs.
Parameter $\epsilon$ allows us some flexibility in applying this heuristic:
A larger $\epsilon$ translates into a higher likelihood that the rounded integer solution will satisfy the desired constraints,
at the expense of missing potential solutions if $\epsilon$ is too large.
Once we identify rounded integer controls and double check that condition (\ref{Eq:SCconvergence}) still holds,
we store these values in a look-up table ready to be implemented in feedback form,
chosen in agreement with the control law proposed in the 
proof of Lemma \ref{Lemma:ConvergenceSet}. 
Having chosen the feedback law, 
we can now also compute the robustly globally attractive control invariant set
$[0,2L_1^*] \times\ \hdots \times [0,2L_6^*]$:
The obtained values are displayed in Table \ref{table:Li},
and are indicated by the dashed blue lines in Figure \ref{fig:Samplepath}.

\begin{table}
\caption{The robustly control invariant box $[0,2L_1^*] \times\ \hdots \times [0,2L_6^*]$.}  
\label{table:Li}
\centering
\begin{tabular}{|c|c|c|c|c|c|}
\hline
 $L_1^*=625$ & $L_2^*=625$ & $L_3^*=221$ & $L_4^*=221$ & $L_5^*=237$ & $L_6^*=237 $ \\\hline
\end{tabular}
\end{table}

Having computed the feedback control law offline,
we report on Monte Carlo simulations of the closed loop system.
We ran 30 different paths with 600 samples (horizon steps) each starting
from uniformly randomly selected initial states 
in the interval $[-1000,2L_i^*+1000]$ and with 
random demand uniformly drawn from the admissible intervals. 
The first sample path starting from initial state $(1242,1543,1281,779,-161,1468)$
is depicted in Fig. \ref{fig:Samplepath},
with the dashed lines delineating the invariant hyperbox.
All simulations are carried out with MATLAB
on an Intel(R) Core(TM)2 Duo CPU P8400 at 2.27 GHz and 3GB of RAM. 
The run time of the offline computation of the control law is less than 5 seconds, 
while the run time of loading all the controls in the look-up table
and running the Monte Carlo simulations is about 12 seconds.

\begin{figure} [htb]
\centering
\includegraphics[width=11cm]{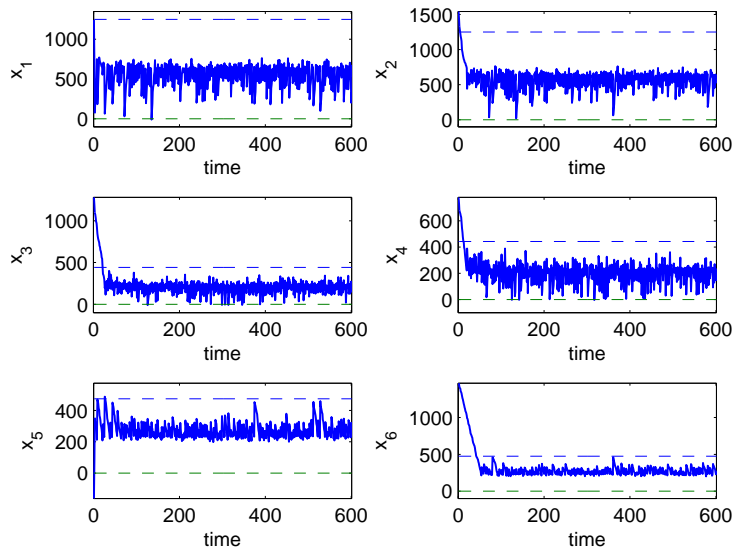}
\caption{First sample path of the Monte Carlo simulations for Example \ref{EX:Two}.}  
\label{fig:Samplepath}
\end{figure}

\section{Conclusions \& Future Work}
\label{Sec:Conclusions}

We considered logistic networks in which the control and disturbance inputs take their values
in finite sets.
We established a necessary and sufficient condition for the existence of robustly control invariant hyperboxes,
and we showed that a stronger version of this condition is sufficient,
albeit not necessary in general, to guarantee robust global attractivity.
The proposed conditions are combinatorial in nature,
which is not surprising as the problem is known to be NP-hard even when the 
input signals are analog.

Future work will focus on deriving bounds on the size of the smallest such invariant hyperboxes,
as well as considering more interesting models of finite alphabet uncertainty.

\section{Acknowledgments}
\label{Sec:Acknowledgments}

D. C. Tarraf's research was supported by NSF CAREER award ECCS 0954601
and AFOSR Young Investigator award FA9550-11-1-0118.
D. Bauso's research was supported by the 2012 ``Research Fellow"
Program of the Dipartimento di Matematica,
Universit\'a di Trento and by the PRIN 20103S5RN3
``Robust decision making in markets and organization."


\bibliographystyle{IEEEtranS}
\bibliography{References}

\end{document}